\documentclass[12pt,oneside,english]{amsart}
\usepackage[T1]{fontenc}
\usepackage[latin9]{inputenc}
\usepackage{color}
\usepackage{fancybox}
\usepackage{calc}
\usepackage{mathrsfs}
\usepackage{amsthm}
\usepackage{amsbsy}
\usepackage{amssymb}
\usepackage{babel}

\usepackage{graphicx}
\usepackage{amsmath}
\usepackage[capitalise]{cleveref}
\crefname{theorem}{Theorem}{Theorems}
\crefname{lemma}{Lemma}{Lemmas}
\crefname{corollary}{Corollary}{Corollaries}
\crefname{proposition}{Proposition}{Propositions}
\crefname{conjecture}{Conjecture}{Conjectures}
\crefname{question}{Question}{Questions}
\crefname{definition}{Definition}{Definitions}
\crefname{example}{Example}{Examples}
\crefname{remark}{Remark}{Remarks}
\crefname{question}{Question}{Questions}
\crefname{enumi}{}{}
\crefname{equation}{}{}

\usepackage[margin=1in]{geometry}

\usepackage[all,cmtip]{xy}

\usepackage{tikz}
\usetikzlibrary{arrows}
\usetikzlibrary{decorations.pathreplacing}

\allowdisplaybreaks

\makeatletter
\numberwithin{equation}{section}
\numberwithin{figure}{section}

  \theoremstyle{plain}
  \newtheorem{theorem}{\protect\theoremname}[section]
  \newtheorem{lemma}[theorem]{\protect\lemmaname}
  \newtheorem{proposition}[theorem]{\protect\propositionname}
  \newtheorem{corollary}[theorem]{\protect\corollaryname}

	\theoremstyle{definition}
  \newtheorem{definition}[theorem]{\protect\definitionname}
	\theoremstyle{remark}
  \newtheorem{remark}[theorem]{\protect\remarkname}

  \theoremstyle{plain}
	\newtheorem*{theorem*}{\protect\theoremname}
  \newtheorem*{lemma*}{\protect\lemmaname}
  \newtheorem*{proposition*}{\protect\propositionname}
  \newtheorem*{corollary*}{\protect\corollaryname}
	\newtheorem{conjecture*}{\protect\conjecturename}
	\newtheorem{question*}{\protect\questionname}
	
	\theoremstyle{definition}
  \newtheorem*{definition*}{\protect\definitionname}
  \theoremstyle{remark}
  \newtheorem*{remark*}{\protect\remarkname}
	
	\providecommand{\theoremname}{Theorem}
  \providecommand{\lemmaname}{Lemma}
  \providecommand{\propositionname}{Proposition}
  \providecommand{\corollaryname}{Corollary}
  \providecommand{\conjecturename}{Conjecture}
	\providecommand{\questionname}{Question}
	\providecommand{\definitionname}{Definition}
  \providecommand{\remarkname}{Remark}

\makeatother

\global\long\def\dee{\mathrm{d}}

\DeclareFontFamily{U}{MnSymbolC}{}
\DeclareSymbolFont{MnSyC}{U}{MnSymbolC}{m}{n}
\DeclareFontShape{U}{MnSymbolC}{m}{n}{
    <-6>  MnSymbolC5
   <6-7>  MnSymbolC6
   <7-8>  MnSymbolC7
   <8-9>  MnSymbolC8
   <9-10> MnSymbolC9
  <10-12> MnSymbolC10
  <12->   MnSymbolC12}{}
\DeclareMathSymbol{\intprod}{\mathbin}{MnSyC}{'270}

\global\long\def\cE{\mathcal{E}}

\global\long\def\ka{\mbox{\large{$\kappa$}}}

\global\long\def\oneb{\bar{1}}

\global\long\def\Ab{\bar{A}}
\global\long\def\Bb{\bar{B}}

\global\long\def\thetah{\hat{\theta}}

\global\long\def\btheta{\boldsymbol{\theta}}
\global\long\def\bh{\boldsymbol{h}}
\global\long\def\Levi{\boldsymbol{L}}

\global\long\def\ups{\upsilon}
\global\long\def\Ups{\Upsilon}

\global\long\def\scrK{\mathscr{K}}

\global\long\def\fg{\mathfrak{g}}

\global\long\def\cT{\mathcal{T}}
\global\long\def\cA{\mathcal{A}}
\global\long\def\cG{\mathcal{G}}

\global\long\def\hook{\lrcorner\,}

\def\sideremark#1{\ifvmode\leavevmode\fi\vadjust{\vbox to0pt{\vss
 \hbox to 0pt{\hskip\hsize\hskip1em
 \vbox{\hsize3cm\tiny\raggedright\pretolerance10000
  \noindent #1\hfill}\hss}\vbox to8pt{\vfil}\vss}}}%

\begin{document}

\title[]{Bounded strictly pseudoconvex domains in $\mathbb{C}^2$ \\ with obstruction flat boundary II}

\author{Sean N.\ Curry}
\author{Peter Ebenfelt}

\thanks{The second author was supported in part by the NSF grant DMS-1600701.}

\begin{abstract}
On a bounded strictly pseudoconvex domain in $\mathbb{C}^n$, $n>1$, the smoothness of the Cheng-Yau solution to Fefferman's complex Monge-Ampere equation up to the boundary is obstructed by a local CR invariant of the boundary. For a bounded strictly pseudoconvex domain $\Omega\subset \mathbb{C}^2$ diffeomorphic to the ball, we prove that the global vanishing of this obstruction implies biholomorphic equivalence to the unit ball, subject to the existence of a holomorphic vector field satisfying a mild approximate tangency condition along the boundary. In particular, by considering the Euler vector field multiplied by $i$ the result applies to all domains in a large $C^1$ open neighborhood of the unit ball in $\mathbb{C}^2$. The proof rests on establishing an integral identity involving the CR curvature of $\partial \Omega$ for any holomorphic vector field defined in a neighborhood of the boundary. The notion of ambient holomorphic vector field along the CR boundary generalizes naturally to the abstract setting, and the corresponding integral identity still holds in the case of abstract CR $3$-manifolds.
\end{abstract}

\maketitle

\section{Introduction}

In this paper we continue our study of compact obstruction flat CR $3$-manifolds, begun in \cite{CE2018-obstruction-flatI}. We recall that if $\Omega\subset \mathbb{C}^n$, $n>1$, is a bounded strictly pseudoconvex domain with smooth boundary $M=\partial \Omega$, then by \cite{ChengYau1980} there is a unique solution $u>0$ to the Dirichlet problem
\begin{equation} \label{eqn:FeffermanComplexMongeAmpere}
\left\{ \begin{array}{l}
\mathcal{J}(u):= (-1)^n \,\mathrm{det} \left( \begin{array}{ c c} u & u_{z^{\bar{k}}}\\ u_{z^{j}} & u_{z^{j}z^{\bar{k}}}\end{array} \right) = 1\;\,\mathrm{in}\;\, \Omega,\\
u=0 \;\,\mathrm{on}\;\, \partial \Omega
\end{array}\right.
\end{equation}
governing the existence of a complete K\"ahler-Einstein metric on $\Omega$ with K\"ahler potential $-\log\, u$.
Cheng and Yau \cite{ChengYau1980} proved that the solution $u$ exists and is smooth in $\Omega$ and $C^{n+\frac{3}{2}-\epsilon}$ up to the boundary, for any $\epsilon >0$. Lee and Melrose \cite{LeeMelrose1982} then showed that $u$ is $C^{n+2-\epsilon}$ up to the boundary by showing that $u$ has asymptotic expansion
\begin{equation}\label{eqn:Lee-Melrose-asymptotics}
u \sim \rho \sum_{k=0}^{\infty} \eta_k (\rho^{n+1}\log \rho)^k, \quad \eta_k \in C^{\infty}(\overline{\Omega})
\end{equation}
where $\rho$ is a Fefferman defining function for $\Omega$, meaning $\rho$ is normalized by $\mathcal{J}(\rho)=1+O(\rho^{n+1})$; see \cite{Fefferman1976}. 
Such a defining function is unique modulo $O(\rho^{n+2})$ and Graham \cite{Graham1987a, Graham1987b} showed that the coefficients $\eta_k$ are local CR invariants modulo $O(\rho^{n+1})$. 
Graham also showed that if $b\eta_1=\eta_1|_{M}$ vanishes, then $\eta_k = 0$ for all $k\geq 1$, and $u$ is smooth up to $M$. Thus $b\eta_1$ is precisely the obstruction to $C^{\infty}$ boundary regularity of the Cheng-Yau solution $u$. 
We say that a strictly pseudoconvex CR manifold is \emph{obstruction flat} if $b\eta_1=0$. 
Graham \cite{Graham1987a, Graham1987b} showed that there is a large family of local (noncompact) strictly pseudoconvex obstruction flat hypersurfaces in $\mathbb{C}^n$, not locally CR equivalent to the unit sphere (not \emph{locally spherical}). In this paper we shall consider the problem of classifying the \emph{compact} strictly pseudoconvex obstruction flat hypersurfaces in $\mathbb{C}^2$. In this case the problem coincides with that of classifying the smooth bounded strictly pseudoconvex domains $\Omega$ for which the trace of the log term in the asymptotic expansion of the Bergman kernel vanishes on $M=\partial\Omega$ \cite{Graham1987b}. In this formulation the problem is subject to a strong form of the classical Ramadanov conjecture \cite{Ramadanov1981}; recall that the Ramadanov conjecture states that the full vanishing of the coefficient of the log term (not just of the trace on the boundary) characterizes the unit ball in $\mathbb{C}^n$. The reader is referred to \cite[Section 2]{CE2018-obstruction-flatI} for more detail. The problem is also equivalent to classifying the compact strictly pseudoconvex hypersurfaces in $\mathbb{C}^2$ that are critical points for the Burns-Epstein invariant, or equivalently for the total $Q'$-curvature, among deformations in $\mathbb{C}^2$ (Kuranishi wiggles); see \cite{CaseYang2013, ChengLee1995, CE2018-obstruction-flatI, Hirachi2014, HirachiMarugameMatsumoto2017}.

In \cite{CE2018-obstruction-flatI} the authors proved that if a compact strictly pseudoconvex CR $3$-manifold with infinitesimal symmetry is obstruction flat then it must be locally spherical, generalizing previous results in \cite{BoichuCoeure1983,Nakazawa1994,Ebenfelt-arxiv2016}. For the case of boundaries in $\mathbb{C}^2$, having an infinitesimal symmetry means the existence of a holomorphic vector field whose real part is tangent to the boundary. In this paper we generalize the aforementioned result of \cite{CE2018-obstruction-flatI} in the embedded case, by substantially relaxing the condition on the tangency of the real part of the holomorphic vector field. 

The restriction of a holomorphic vector field $X$ to a strictly pseudoconvex hypersurface $M$ in $\mathbb{C}^2$ is determined by a single complex function (or density) $u$ given by $u=\overline{iX\rho}$ where $\rho$ is a defining function for $M$; the function $\bar{u}=iX\rho$ is a kind of Hamiltonian potential for $X$ (see \cref{lem:solution-from-holomorphic-vf}). The holomorphic vector field $X$ gives rise to an infinitesimal symmetry of $M$ precisely when $u$ is real, which means that the real part of $X$ is tangent to $M$.  For our main result, we shall need a much weaker condition on $u$; we merely require that the imaginary part of $u$ is not too large relative to the real part. We make the following definition. 
\begin{definition}
Let $M$ be a smooth hypersurface in $\mathbb{C}^2$, and let $\rho$ be any defining for $M$ in a neighborhood of $p\in M$. Let $\epsilon \geq 0$. We say that the real part of a $(1,0)$-vector field $X$ is \emph{$\epsilon$-approximately tangent} to $M$ at $p$ if $u=\overline{iX\rho}$ satisfies
$$
(\mathrm{Im}\,u)^2 \leq \epsilon \,(\mathrm{Re}\,u)^2
$$
at $p$. We say that $\mathrm{Re}\,X$ is \emph{strictly $\epsilon$-approximately tangent} at $p$ if the above inequality is strict.
\end{definition}
\begin{remark} (i) It is easy to see that the notion of $\epsilon$-approximate tangency does not depend on the choice of defining function $\rho$.\\
(ii) Note that if $X$ is a $(1,0)$-vector field, then $\mathrm{Re}\, X$ is tangent to $M$ if and only if $\mathrm{Re}\, X$ is $0$-approximately tangent to $M$.\\
(iii) If $u(p)\neq 0$, then strict $\epsilon$-approximate tangency means $|\mathrm{arg}\,u|<\alpha$ or $|\mathrm{arg}\,{(-u)}|<\alpha$ at $p$ with $\alpha=\arctan\sqrt{\epsilon}$; a similar condition was introduced by Barrett in his study of Sobolev regularity for the Bergman projection on bounded domains \cite{Barrett1986}.
\end{remark}

\begin{theorem}\label{thm:main-theorem1}
Let $\Omega\subset \mathbb{C}^2$ be a smooth bounded strictly pseudoconvex domain for which there exists a holomorphic vector field $X$ on $\overline{\Omega}$ whose real part is strictly $1$-approximately tangent to $\partial \Omega$ almost everywhere.
If $\partial \Omega$ is obstruction flat, then $\partial \Omega$ is locally spherical. 
\end{theorem}

Taking $X$ to be $i$ times the Euler vector field, \cref{thm:main-theorem1} applies to all domains in a large $C^1$ open neighborhood of the unit ball. More precisely, we define $\mathscr{D}$ to be the space of smooth embedded strictly pseudoconvex $3$-spheres in $\mathbb{C}^2$ (with the $C^{\infty}$ topology). By identifying an element $M\in\mathscr{D}$ with the domain $\Omega\subset \mathbb{C}^2$ that it bounds, we think of $\mathscr{D}$ as the space of strictly pseudoconvex domains $\Omega\subset\mathbb{C}^2$ with smooth boundary $\partial \Omega$ diffeomorphic to $S^3$. Let $\mathscr{U}\subset \mathscr{D}$ denote the open neighborhood of the unit ball $\mathbb{B}^2$ consisting of domains for which $X$ is strictly $1$-approximately tangent to the boundary; this holds if the real part of the Euler field is closer to being normal than tangent to the boundary of the domain. 

\begin{corollary}
For all domains $\Omega$ in the neighborhood $\mathscr{U}$ of the unit ball $\mathbb{B}^2 \subset \mathbb{C}^2$, if $\partial\Omega$ is obstruction flat, then $\Omega$ is biholomorphic to the unit ball.
\end{corollary}
\begin{remark}
While $\mathscr{D}$ is naturally equipped with the $C^{\infty}$ topology, by construction the set $\mathscr{U}$ is clearly also open in the (relative) $C^1$ topology.
\end{remark}

The conjugate $u$ of the Hamiltonian potential for a holomorphic vector field $X$ satisfies the differential equation
\begin{equation}\label{eqn:infinitesimal-automorphism-equation}
\nabla_1\nabla_1 u +iA_{11}u=0
\end{equation}
on the CR manifold $M\subset \mathbb{C}^2$, where $\nabla$ denotes the Tanaka-Webster connection of some contact form and $A_{11}$ is its pseudohermitian torsion; the operator $\nabla_1\nabla_1 +iA_{11}$ is in fact CR invariant when acting on CR densities of weight $(1,1)$, cf.\ \cref{sec:weighted-pseudohermitian-calculus}, and the corresponding equation makes sense on an abstract CR manifold. We have the following more general result, which implies \cref{thm:main-theorem1}.

\begin{theorem}\label{thm:main-theorem2}
Let $(M,H,J)$ be a compact strictly pseudoconvex CR $3$-manifold for which there exists a weight $(1,1)$ complex solution $u$ of \cref{eqn:infinitesimal-automorphism-equation} such that $(\mathrm{Im}\,u)^2 < (\mathrm{Re}\,u)^2$ almost everywhere. If $(M,H,J)$ is obstruction flat, then it is locally spherical.
\end{theorem}

This theorem will follow easily from:
\begin{theorem}\label{thm:main-technical-theorem}
Let $(M,H,J)$ be a compact strictly pseudoconvex CR $3$-manifold which is obstruction flat. For any weight $(1,1)$ complex solution $u$ of \cref{eqn:infinitesimal-automorphism-equation} we have
$$
\int_M u^2 |Q|^2 = 0.
$$
\end{theorem}
\begin{remark}
Note that since $u$ has weight $(1,1)$ and $|Q|^2=\bh^{1\oneb}\bh^{1\oneb}Q_{11}Q_{\oneb\oneb}$ has weight $(-4,-4)$, the integrand $u^2 |Q|^2$ has weight $(-2,-2)$ and therefore may be naturally identified with a complex volume form on $M$. See \cref{sec:weighted-pseudohermitian-calculus} for a discussion of weights and densities.
\end{remark}
The main technical ingredient in the proof of this result is \cref{prop:holomorphic-prolonged-system}, which constructs from a solution $u$ of \cref{eqn:infinitesimal-automorphism-equation} a solution of a prolonged system of $\bar{\partial}_b$-equations involving the CR curvature.

\subsection{Outline of the Paper}
In \cref{sec:pseudohermitian-calculus} we recall some background on CR geometry and weighted pseudohermitian calculus, and make some observations concerning complex solutions of the infinitesimal automorphism equation \cref{eqn:infinitesimal-automorphism-equation}. In \cref{sec:existence-of-solutions} we establish the connection between weight $(1,1)$ complex solutions of the infinitesimal automorphism equation on an embedded CR manifold and ambient holomorphic vector fields. In \cref{sec:tractor-calculus-and-prolongation} we introduce the CR tractor calculus and prove \cref{prop:holomorphic-prolonged-system}. In \cref{sec:main-technical-theorem} we prove our main results, \cref{thm:main-theorem2,thm:main-technical-theorem}.

\section{Strictly pseudoconvex CR $3$-manifolds}\label{sec:pseudohermitian-calculus}

For the reader's convenience, we here recall the general setup from \cite{CE2018-obstruction-flatI}. A \emph{strictly pseudoconvex CR 3-manifold} is a triple $(M,H,J)$ where $M$ is a smooth oriented $3$-manifold, $H\subset TM$ is a contact distribution, and $J:H\to H$ is a smooth bundle endomorphism such that $J^2=-\mathrm{id}$. The partial complex structure $J$ on $H\subset TM$ defines an orientation of $H$, and therefore defines an orientation on the annihilator subbundle $H^{\perp}:=\mathrm{Ann}(H)\subset T^*M$. Given any contact form $\theta$ for $H$, $d\theta |_H$ is a nondegenerate bilinear form. A contact form $\theta$ for $H$ is positively oriented if $\dee \theta( \,\cdot\, , J\,\cdot\,)$ is positive definite on $H$. A strictly pseudoconvex CR structure $(M,H,J)$ together with a choice of positively oriented contact form $\theta$ is referred to as a \emph{pseudohermitian structure}. The \emph{Reeb vector field} of a contact form $\theta$ is the vector field $T$ uniquely determined by $\theta(T)=1$ and $T\intprod \dee\theta =0$. 

Given a CR manifold $(M,H,J)$ we decompose the complexified contact distribution $\mathbb{C}\otimes H$
as $T^{1,0}\oplus T^{0,1}$, where $J$ acts by $i$ on $T^{1,0}$ and by $-i$ on $T^{0,1}=\overline{T^{1,0}}$. Let $\theta$ be an oriented contact form on $M$. Let $Z_1$ be a local frame for the \emph{holomorphic tangent bundle} $T^{1,0}$ and $Z_{\oneb}=\overline{Z_1}$, so that $\{T,Z_1,Z_{\oneb}\}$ is a local frame for $\mathbb{C}\otimes TM$. Then the dual frame $\{\theta,\theta^1,\theta^{\oneb}\}$ is referred to as an \emph{admissible coframe} and one has
\begin{equation}\label{eqn:h11-definition}
\dee \theta = i h_{1\oneb} \theta^1 \wedge \theta^{\oneb}
\end{equation}
for some positive smooth function $h_{1\oneb}$. The function $h_{1\oneb}$ is the component of the Levi form $\mathrm{L}_{\theta}(U,\overline{V})=-2i\dee\theta(U,\overline{V})$ on $T^{1,0}$, that is
\begin{equation*}
\mathrm{L}_{\theta}(U^1Z_1,V^{\oneb}Z_{\oneb}) = h_{1\oneb}U^1V^{\oneb}.
\end{equation*}
It is sometimes convenient to scale $Z_1$ so that $h_{1\oneb}=1$, but we will not assume this. We write $h^{1\oneb}$ for the multiplicative inverse of $h_{1\oneb}$. The Tanaka-Webster connection associated to $\theta$ is given in terms of such a local frame $\{T,Z_1,Z_{\oneb}\}$ by
\begin{equation*}
\nabla Z_1 = \omega_1{}^{1}\otimes Z_1, \quad \nabla Z_{\oneb} = \omega_{\oneb}{}^{\oneb}\otimes Z_{\oneb}, \quad \nabla T =0
\end{equation*}
where the connection $1$-forms $\omega_1{}^{1}$ and $\omega_{\oneb}{}^{\oneb}$ satisfy
\begin{equation}\label{eqn:pseudohermitian-connection1}
\dee \theta^1 = \theta^1\wedge \omega_1{}^{1} + A^1{}_{\oneb}\,\theta\wedge\theta^{\oneb}, \text{ and}
\end{equation}
\begin{equation}\label{eqn:pseudohermitian-connection2}
\omega_1{}^{1} + \omega_{\oneb}{}^{\oneb} =h^{1\oneb}\dee h_{1\oneb},
\end{equation}
for some function $A^1{}_{\oneb}$. The uniquely determined function $A^1{}_{\oneb}$ is known as the \emph{pseudohermitian torsion}. Components of covariant derivatives will be denoted by adding $\nabla$ with an appropriate subscript, so, e.g., if $u$ is a function then $\nabla_1 u = Z_1 u$ and $\nabla_0\nabla_1 u = T Z_1 u - \omega_1{}^{1}(T)Z_1u$. We may also use $h_{1\oneb}$ and $h^{1\oneb}$ to raise and lower indices, so that $A_{\oneb\oneb}= h_{1\oneb}A^1{}_{\oneb}$ and $A_{11}=h_{1\oneb}A^{\oneb}{}_{1}$, with $A^{\oneb}{}_{1} =\overline{A^1{}_{\oneb}}$.

We recall some useful formulae, which can be found in \cite{CurryGover-arxiv2015, Lee1988}. The \emph{pseudohermitian (scalar) curvature} $R$ is defined by the structure equation
\begin{equation*}
\dee \omega_1{}^{1} = Rh_{1\oneb} \theta^1\wedge \theta^{\oneb} + (\nabla^1 A_{11})\,\theta^1\wedge\theta - (\nabla^{\oneb} A_{\oneb\oneb})\,\theta^{\oneb}\wedge\theta.
\end{equation*}
The torsion of the Tanaka-Webster connection (as an affine connection) is captured by the following formulae, for a smooth function $f$,
\begin{equation*}
\nabla_1\nabla_{\oneb}f-\nabla_{\oneb}\nabla_{1}f  =-i h_{1\oneb}\nabla_{0}f,\quad \text{ and }\quad  \nabla_{1}\nabla_{0}f-\nabla_{0}\nabla_{1}f  =A^{\oneb}{_{1}}\nabla_{\oneb}f.
\end{equation*}
 The pseudohermitian curvature $R$ may therefore equivalently be defined by the Ricci identity
\begin{equation}\label{eqn:Ricci-identity}
\nabla_1\nabla_{\oneb} V^1 -\nabla_{\oneb}\nabla_1 V^1 + i h_{1\oneb}\nabla_0 V^1 = Rh_{1\oneb} V^1
\end{equation}
for any local section $V^1Z_1$ of $T^{1,0}$. Commuting $0$ and $1$ (or $\oneb$) derivatives on $V^1Z_1$ gives torsion according to the following formulae
\begin{equation}\label{eqn:Tanaka-Webster-commuting-10-derivatives}
\nabla_1\nabla_{0} V^1 -\nabla_{0}\nabla_1 V^1 - A^{\oneb}{_{1}}\nabla_{\oneb}V^1 = (\nabla^1 A_{11}) V^1, \;\;\text{and}
\end{equation}
\begin{equation*}
\nabla_{\oneb}\nabla_{0} V^1 -\nabla_{0}\nabla_{\oneb} V^1 - A^{1}{_{\oneb}}\nabla_{1}V^1 = (\nabla^{\oneb} A_{\oneb\oneb}) V^1.
\end{equation*}
In dimension $3$, the Bianchi identities of \cite[Lemma 2.2]{Lee1988} reduce to
\begin{equation}\label{eqn:pseudohermitian-Bianchi}
\nabla_0 R = 2\mathrm{Re}\,(\nabla^1 \nabla^1 A_{11}).
\end{equation}

\subsection{Weighted pseudohermitian calculus}\label{sec:weighted-pseudohermitian-calculus}

Let $(M,H,J)$ be a CR $3$-manifold, and let $\Lambda^{1,0}$ denote the complex rank $2$ bundle of $(1,0)$-forms on $M$. The bundle $\Lambda^{2,0}=\Lambda^2(\Lambda^{1,0})$ of $(2,0)$-forms is referred to as the \emph{canonical line bundle} of $M$, and denoted by $\scrK$. We assume throughout that its dual $\scrK^*$ admits a (global) cube root, which we fix and denote by $\cE(1,0)$. Note that this is equivalent to assuming that the integral first Chern class of $\scrK$ is divisible by $3$; in particular this holds for hypersurfaces in $\mathbb{C}^2$. We then define the \emph{CR density line bundle} of weight $(w,w')$ to be $\cE(w,w') = \cE(1,0)^w \otimes \overline{\cE(1,0)}^{w'}$, where $w,w'\in \mathbb{C}$ with $w-w'\in \mathbb{Z}$. Note that for $w$ real the bundle $\cE(w,w)$ is invariant under conjugation, and hence contains a real subbundle $\cE_{\mathbb{R}}(w,w)$. Note also that by definition $\cE(3,0)=\scrK^*$, so $\cE(-3,0)=\scrK$.

Trivializing the bundle $TM/H$ determines a contact form on $M$ via the natural map $TM\rightarrow TM/H$.
Similarly, a choice of non-vanishing section $\zeta$ (i.e.\ a trivialization) of $\scrK$ determines canonically a contact form $\theta$ on $M$ by the requirement \cite{Farris1986} (see also \cite{Lee1988}) that
\begin{equation}\label{eq:volume-normalized}
\theta\wedge \dee\theta = i\theta \wedge (T\intprod \zeta)\wedge (T\intprod\overline{\zeta}).
\end{equation}
In this case we say that $\theta$ is \emph{volume normalized} with respect to $\zeta$. Combining these observations, we may realize $TM/H$ as a real CR density line bundle as follows. A contact form $\theta$ determines canonically a section $|\zeta|^2=\zeta\otimes\overline{\zeta}$ of $\scrK\otimes\overline{\scrK} = \cE(-3,-3)$ by the condition that $\zeta$ satisfy \cref{eq:volume-normalized} ($\zeta$ is only determined up to phase at each point). If we rescale $\theta$ to $\thetah=e^{\Ups}\theta$, with $\Ups\in C^{\infty}(M,\mathbb{R})$, then the corresponding section $|\hat{\zeta}|^2$ equals $e^{3\Ups}|\zeta|^2$. Thus, the map which assigns to a contact form $\theta$ the section $|\zeta|^{2/3}$ of $\cE_{\mathbb{R}}(-1,-1)$ extends to a canonical isomorphism of $H^{\perp}$ with $\cE_{\mathbb{R}}(-1,-1)$. Dually $TM/H$ is canonically isomorphic to $\cE_{\mathbb{R}}(1,1)$. This identification gives us a tautological $1$-form $\btheta$ of weight $(1,1)$, corresponding to the map $TM\to TM/H=\cE_{\mathbb{R}}(1,1)$.

We define the \emph{CR Levi form} $\Levi:T^{1,0}\otimes T^{0,1} \to \mathbb{C}TM/\mathbb{C}H$ by
\begin{equation*}
\Levi(U,\overline{V}) = 2i[U,\overline{V}]\;\, \mathrm{mod}\; \mathbb{C}H.
\end{equation*}
On a strictly pseudoconvex CR $3$-manifold the CR Levi form is a bundle isomorphism, so we have $T^{1,0}\otimes T^{0,1} \cong \mathbb{C}TM/\mathbb{C}H = \cE(1,1)$. The CR Levi form may be interpreted as a Hermitian bundle metric on $T^{1,0}\otimes\cE(-1,0)$, and we would like to have a more concise notation for bundles like this one. We use the symbol $\cE$ decorated with appropriate indices to denote the tensor bundles constructed from $T^{1,0}$ and $T^{0,1}$ (this is Penrose's abstract index notation). For example, $\cE^1=T^{1,0}$, $\mathcal{E}_{\oneb}=(T^{0,1})^*$, and $\mathcal{E}_{1\oneb}=(T^{1,0})^*\otimes(T^{0,1})^*$. We will now generally use abstract index notation for sections of these bundles. So, for example, $V^1$ may denote a global section of $\cE^1=T^{1,0}$ (previously written locally as $V^1Z_1$). This keeps the notation from getting too heavy, and allows us to globalize our previous local formulas. Note that a choice of contact form allows us to decompose the complexified tangent bundle $\mathbb{C}TM$ as $\cE^1\oplus\cE^{\oneb}\oplus\cE(1,1)$. Using abstract index notation we may therefore decompose $V$ globally as $V\overset{\theta}{=} (V^1,V^{\oneb},V^0)$. Generally we denote the tensor product of a complex vector bundle $\mathcal{V}$ on $M$ with $\cE(w,w')$ by appending $(w,w')$, as in $\mathcal{V}(w,w')$. The CR Levi form will be thought of as a section $\bh_{1\oneb}$ of $\cE_{1\oneb}(1,1)$, with inverse $\bh^{1\oneb}$. The Levi form will be used to identify $\cE^{1\oneb}$ with $\cE(1,1)$, and $\cE_{1\oneb}$ with $\cE(-1,-1)$, and to raise and lower indices without comment. 

Observe that the Tanaka-Webster connection $\nabla$ of a pseudohermitian structure $\theta$ extends naturally to act on the CR density bundles, since $\nabla$ acts on the canonical bundle $\scrK$. Since the Tanaka-Webster connection of $\theta$ preserves $\theta$, and also preserves the section $|\zeta|^2$ of $\scrK\otimes\overline{\scrK} = \cE(-3,-3)$ determined by volume normalization, the Tanaka-Webster connection respects the CR invariant identification of $TM/H$ with $\cE_{\mathbb{R}}(1,1)$; see \cite{CurryGover-arxiv2015,GoverGraham2005}. Another way of saying this is that $\nabla \btheta = 0$. A similar argument shows that $\nabla$ preserves the CR Levi form, $\nabla \Levi =0$. Hence, the Tanaka-Webster connection of $\theta$ respects all of the CR invariant identifications made above. We therefore make use of CR densities whenever convenient.

We will need to commute derivatives of weighted tensor fields, for this we need to know the curvature of the CR density bundles. Let $\tau$ be a section of $\cE(w,w')$. From \cref{eqn:Ricci-identity,eqn:Tanaka-Webster-commuting-10-derivatives} one easily obtains (cf.\ \cite[Proposition 2.2]{GoverGraham2005}) that
\begin{align}
\label{eqn:commuting-11bar-derivatives-on-densities}
\nabla_1\nabla_{\oneb} \tau -\nabla_{\oneb}\nabla_1 \tau + i \bh_{1\oneb}\nabla_0 \tau & = \frac{w-w'}{3}R \bh_{1\oneb}\, \tau \,; \\
\label{eqn:commuting-10-derivatives-on-densities}
\nabla_1\nabla_{0} \tau -\nabla_{0}\nabla_1 \tau - A^{\oneb}{_{1}}\nabla_{\oneb}\tau & = \frac{w-w'}{3}(\nabla_{\oneb} A^{\oneb}{}_1) \tau.
\end{align}

\subsection{CR invariants}

The local calculus on CR manifolds associated with the CR Cartan connection is discussed in more detail in Section 7 of \cite{CE2018-obstruction-flatI}. For now it suffices to recall some basic definitions and formulae in terms of pseudohermitian calculus. The \emph{Cartan umbilical tensor} $Q$ of $(M,H,J)$ is a (weighted) CR invariant, whose vanishing in a neighborhood is necessary and sufficient for $(M,H,J)$ to be locally equivalent to the induced CR structure on the unit sphere in $\mathbb{C}^2$. As in \cite{ChengLee1990}, given a choice of contact form $\theta$ we interpret the umbilical tensor $Q$ as an endomorphism of $H$, written locally as
\begin{equation}\label{eqn:Q-endomorphism}
Q=iQ_1{}^{\oneb}\theta^1\otimes Z_{\oneb} - iQ_{\oneb}{}^1\theta^{\oneb}\otimes Z_1.
\end{equation}
By \cite[Lemma 2.2]{ChengLee1990} the component $Q_{11}$ of Cartan's tensor is given by
\begin{equation}\label{eqn:Cartan-umbilical-tensor}
Q_{11} =-\frac{1}{6}\nabla_1\nabla_1 R - \frac{i}{2}RA_{11} + \nabla_0 A_{11} +\frac{2i}{3} \nabla_1\nabla^1 A_{11},
\end{equation}
where we have taken the opposite sign convention. If $\thetah=e^{\Ups}\theta$ is another contact form, then $\hat{Q}=e^{-2\Ups}Q$, so that $Q$ may be thought of more invariantly as a weighted section of $\mathrm{End}(H)$. More precisely, $Q$ may be thought of as a CR invariant section of $\mathrm{End}(H)\otimes \cE_{\mathbb{R}}(-2,-2)$, the dependency on the contact form $\theta$ only being introduced when we use $\theta$ to trivialize $\cE_{\mathbb{R}}(1,1)=TM/H$. Note that this means $Q_{11}$ is a section of $\cE_{11}(-1,-1)$. The Bianchi identity for the curvature of the CR Cartan connection is equivalent to the following Bianchi identity for $Q$, 
\begin{equation}\label{eqn:Q11-Bianchi}
\mathrm{Im}(\nabla^1\nabla^1Q_{11}-iA^{11}Q_{11})=0,
\end{equation}
which may also be seen as a direct consequence of \cref{eqn:pseudohermitian-Bianchi}. The \emph{CR obstruction density} is given by
\begin{equation}\label{eqn:obstruction-density}
\mathcal{O}=\frac{1}{3}(\nabla^1\nabla^1Q_{11}-iA^{11}Q_{11}).
\end{equation}
The CR obstruction density $\mathcal{O}$ is again a (weighted) CR invariant. If $\thetah=e^{\Ups}\theta$ is another contact form, then $\hat{\mathcal{O}}=e^{-3\Ups}\mathcal{O}$, so that $\mathcal{O}$ defines a CR invariant section of $\cE_{\mathbb{R}}(-3,-3)$.

\subsection{Complex solutions to the infinitesimal automorphism equation}
Central to our considerations will be the operator 
\begin{equation}
\nabla_1\nabla_1 + iA_{11} : \cE(1,1)\to\cE_{11}(1,1)
\end{equation}
whose formal adjoint, with respect to the canonical CR invariant weight $(2,2)$ volume form on $M$, is the operator $\nabla^1\nabla^1-iA^{11}:\cE_{11}(-1,-1)\to \cE(-3,-3)$ appearing in \cref{eqn:obstruction-density}. It is well known that a CR manifold $M$ possesses an infinitesimal CR automorphism if and only if there is a real solution to the \emph{infinitesimal automorphism equation} 
$$
\nabla_1\nabla_1 f + iA_{11}f =0,
$$ 
in which case the infinitesimal symmetry is given by the contact Hamiltonian vector field $V_f = fT + i f^1 Z_1 - i f^{\oneb}Z_{\oneb}$ with potential $f$; see, e.g., \cite{CE2018-obstruction-flatI}. For our purposes, it is the existence of complex solutions of this equation that will turn out to be essential, and this is intimately related with the question of local embeddability of the CR manifold $(M,H,J)$ in $\mathbb{C}^2$.
\begin{proposition}\label{prop:embeddability}
Let $(M,H,J)$ be a strictly pseudoconvex CR manifold and $p\in M$. The following are equivalent:
\begin{itemize}
\item[(i)] $(M,H,J)$ is embeddable in a neighborhood of $p$.
\item[(ii)] There is a nonvanishing complex solution $u$ of the equation 
\begin{equation*}
\nabla_1\nabla_1 u + iA_{11}u =0
\end{equation*}
 on weight $(1,1)$ densities in a neighborhood of $p$.
\end{itemize}
\end{proposition}
This proposition follows from the following three lemmas. 
\begin{lemma}\label{lem:Lie_V-of-Z_1}
Let $(M,H,J)$ be a strictly pseudoconvex CR manifold and $\theta$ a contact form for $H$. Let $f$ be a complex function on $M$ and let $V=fT+if^1Z_1$, where $T$ is the Reeb vector field of $\theta$, $Z_1$ spans $T^{1,0}M$, and $f^1=\nabla^1f$. Then
$$
\mathcal{L}_{V} Z_{\oneb} = -i(\nabla_{\oneb}\nabla^{1}f-iA_{\oneb}{}^{1}f)Z_1  \mod T^{0,1}M.
$$
In particular, if $\nabla_{\oneb}\nabla_{\oneb}f-iA_{\oneb\oneb}f = 0$, then $\mathcal{L}_{V} \Gamma(T^{0,1}M) \subset \Gamma(T^{0,1} M)$. 
\end{lemma}
\begin{remark}
By $\Gamma(T^{0,1}M)$ we mean the space of smooth sections of $T^{0,1}M$.
\end{remark}
\begin{proof}
Let $(\theta, \theta^1,\theta^{\oneb})$ be the admissible coframe dual to $(T,Z_1,Z_{\oneb})$. Using the structure equation \cref{eqn:pseudohermitian-connection1} and $\theta^1(V)=if^1$ we have
\begin{align*}
\mathcal{L}_V \theta^1 &= V\hook \dee\theta^1 + \dee(V\hook \theta^1)\\
&= fA^1{}_{\oneb}\theta^{\oneb} +i(f^1{}_{\oneb}\theta^{\oneb}+f^1{}_1\theta^1+f^1{}_0\theta).
\end{align*}
So $\theta^1(\mathcal{L}_V Z_{\oneb}) = - (\mathcal{L}_V \theta^1)(Z_{\oneb}) = -i(\nabla_{\oneb}\nabla^{1}f-iA_{\oneb}{}^{1}f)$. Moreover, $\mathcal{L}_V \theta = V\hook \dee\theta + \dee(V\hook \theta) = -\bh_{1\oneb}f^1\theta^{\oneb} + df$, so that $\theta(\mathcal{L}_V Z_{\oneb}) = - (\mathcal{L}_V \theta)(Z_{\oneb}) = 0$, which gives the desired result.
\end{proof}

\begin{lemma} 
Let $(M,H,J)$ be a strictly pseudoconvex CR manifold, and let $V$ be a complex vector field on $M$ with the property that $\mathcal{L}_{V} \Gamma(T^{0,1}M) \subset \Gamma(T^{0,1} M)$. Then $V=fT+if^1Z_1\mod T^{0,1}M$ where $f^1=\nabla^1f$ and $u=\overline{f}$ satisfies the equation $\nabla_1\nabla_1 u + iA_{11}u =0$.
\end{lemma}
\begin{proof} This is a straightforward calculation along the lines of the previous proof.
\end{proof}
The following lemma of Jacobowitz \cite{Jacobowitz1987} now establishes \cref{prop:embeddability}.
\begin{lemma}{\cite[Proposition 2.1]{Jacobowitz1987}} Let $(M,H,J)$ be a strictly pseudoconvex CR manifold. The following are equivalent:
\begin{itemize}
\item[(i)] $(M,H,J)$ is embeddable in a neighborhood of the point $p$.
\item[(ii)] There exists a vector field $V$ with $\mathcal{L}_{V} \Gamma(T^{0,1}M) \subset \Gamma(T^{0,1} M)$ and  $V\notin T^{1,0}M\oplus T^{0,1}M$ in a neighborhood of $p$.
\end{itemize}
\end{lemma}

\section{Existence of global complex solutions to the infinitesimal automorphism equation}\label{sec:existence-of-solutions}

While the space of real solutions to the infinitesimal automorphism equation on an embedded strictly pseudoconvex CR manifold is finite dimensional, and generically consists only of the zero solution, the space of complex solutions is always infinite dimensional (corresponding to the infinite dimensional space of ambient holomorphic vector fields). Let $M\subset\mathbb{C}^2$ be a strictly pseudoconvex hypersurface and $\theta$ a contact form for $M$. Given a section $X\in \Gamma(T^{1,0}\mathbb{C}^2|_M)$, expressed in terms of an admissible frame by  
$$
X = X^0 \xi  + X^1Z_1
$$
with $\xi=\frac{1}{2}(T-iJT)$ where here $J$ denotes the standard complex structure on $\mathbb{C}^2$, we define the complex vector field $X_M$ on $M$ by
\begin{equation}\label{eqn:X_M}
X_M = X^0 T + X^1Z_1.
\end{equation}
Note that the vector field $X_M$ depends on the choice of $\theta$, but is well defined modulo $T^{0,1}M$.
\begin{lemma}\label{lem:solution-from-holomorphic-vf}
Let $M\subset\mathbb{C}^2$ be a strictly pseudoconvex hypersurface and $\theta$ a contact form for $M$. If $X$ is a holomorphic vector field defined on a neighborhood of $M$ then $X_M=fT+if^1 Z_1$ where $u=\overline{f}$ satisfies
$$\nabla_1\nabla_1 u +iA_{11} u =0.$$ 
\end{lemma}
\begin{proof}
Let $\phi_t$ denote the flow of $\mathrm{Re}\,X$, and let $\psi_t=\phi_t|_M$ be the resulting parametrized deformation of $M\subset \mathbb{C}^2$, which is trivial since the $\phi_t$ are biholomorphic. Then the variational vector field (interpreted as a section of $T_{(1,0)}M=\mathbb{C}TM/T^{0,1}M$) is given by $\dot{\psi} = X_M \;\mathrm{mod}\; T^{0,1}$. Hence, by Lemma 4.5 of \cite{CE2018-obstruction-flatI}, we have $X_M = fT+if^1 Z_1$ where $f=\theta(X_M)$ and the infinitesimal deformation tensor $\varphi_{11}$ is given by $\varphi_{11}=i(\nabla_1\nabla_1 u +iA_{11} u)$ with $u=\overline{f}$. But $\psi_t$ is a trivial deformation, so $\varphi_{11}=0$. The result follows.
\end{proof}
It follows from \cref{lem:Lie_V-of-Z_1} and \cref{lem:solution-from-holomorphic-vf} that if $X$ is a holomorphic vector field then $V=X_M$ satisfies $\mathcal{L}_{V} \Gamma(T^{0,1}M) \subset \Gamma(T^{0,1} M)$.

\begin{lemma}\label{lem:holomorphic-vf-from-solution}
Let $M\subset\mathbb{C}^2$ be a strictly pseudoconvex hypersurface. Suppose $u$ is a weight $(1,1)$ complex solution to $\nabla_1\nabla_1 u +iA_{11} u =0$. Then there is an ambient $(1,0)$-vector field $X=a^1\frac{\partial}{\partial z^1} + a^2\frac{\partial}{\partial z^2}$ along $M$ with CR coefficients $a^1$, $a^2$, such that $u=\overline{\btheta(X_M)}$. 
\end{lemma}
\begin{proof}
Let $\theta$ be a contact form for $M$ and define $V=fT +if^1Z_1$ where $f=\overline{u}$, which we now think of as a complex function on $M$. Let $X$ be the ambient $(1,0)$-vector field along $M$ determined by $X_M=V$, where $X_M$ is defined as in \cref{eqn:X_M}. Then $X=a^1\frac{\partial}{\partial z^1} + a^2\frac{\partial}{\partial z^2}$ where 
$$
a^k=dz^k(X)=dz^k(X_M) = f\dee z^k(T) + if^1 \dee z^k(Z_1),
$$
for $k=1,2$. The result can then be gleaned from the proof of \cite[Lemma 4.5]{CE2018-obstruction-flatI}, but we include a proof here for the readers convenience. Noting that $\dee z^k(T)=Tz^k$ and $\dee z^k(Z_1) = Z_1 z^k$, we compute
\begin{align*}
Z_{\oneb} a^k &= (Z_{\oneb}f)Tz^k + f Z_{\oneb}Tz^k + i(Z_{\oneb}f^1)Z_1z^k + if^1 Z_{\oneb}Z_1z^k\\
\nonumber &= (Z_{\oneb}f)Tz^k + f [Z_{\oneb},T]z^k + i(Z_{\oneb}f^1)Z_1z^k + if^1 [Z_{\oneb},Z_1]z^k.
\end{align*}
From the structure equations \cref{eqn:h11-definition} and \cref{eqn:pseudohermitian-connection1} it is straightforward to compute that
\begin{equation*}
[Z_{\oneb},Z_1]= ih_{1\oneb}T + \omega_1{}^1(Z_{\oneb})Z_1 - \omega_{\oneb}{}^{\oneb}(Z_1)Z_{\oneb},
\quad \mathrm{and} \quad
[Z_{\oneb},T] = A^1{}_{\oneb}Z_1 - \omega_{\oneb}{}^{\oneb}(T)Z_{\oneb}.
\end{equation*}
We therefore obtain
\begin{align}\label{eqn:Zoneb-ak}
Z_{\oneb} a^k &= (Z_{\oneb}f - h_{1\oneb}f^1)Tz^k + (iZ_{\oneb}f^1 + f A^1{}_{\oneb} + if^1\omega_1{}^1(Z_{\oneb}) )Z_1z^k \\
\nonumber &= i(\nabla_{\oneb}\nabla^1 f - iA_{\oneb}{}^1f) Z_1z^k =0
\end{align}
for $k=1,2$.
\end{proof}
\begin{remark}
From the first line of \cref{eqn:Zoneb-ak} above and the independence of the vectors $T$ and $Z_1$ it follows that in the statement of \cref{lem:solution-from-holomorphic-vf} it is sufficient for the $(1,0)$-vector field to be defined only on $M$ and have CR coefficients $a^1$ and $a^2$.
\end{remark}

\begin{remark}
Note that by \cref{prop:embeddability,lem:holomorphic-vf-from-solution} if $(M,H,J)$ is an abstract CR $3$-manifold, and $u$ is a complex solution of the infinitesimal automorphism equation which does not vanish at some point, then locally $M$ can be realized as a hypersurface in $\mathbb{C}^2$ such that $f=\bar{u}$ is the complex Reeb component of a holomorphic vector field. More precisely, if $(M,H,J)$ is an abstract CR $3$-manifold, and $u$ is a complex solution of the infinitesimal automorphism equation with $u(p)\neq 0$, then there is a CR embedding $\psi$ from a neighborhood $U$ of $p$ into $\mathbb{C}^2$ and a holomorphic vector field $X$ on the pseudoconvex side of $\psi(U)$ extending smoothly to $\psi(U)$ such that $\btheta(X_{\psi(U)})=\bar{u}$.
\end{remark}

\section{CR sections of the complexified adjoint tractor bundle} \label{sec:tractor-calculus-and-prolongation}

\subsection{Tractor calculus}

Here we recall the setup for the CR tractor calculus in $3$-dimensions, as in \cite{CE2018-obstruction-flatI}. Let $\mathbb{C}^{2,1}$ denote the defining representation of $\mathrm{SU}(2,1)$. Let $P$ be the subgroup of $G=\mathrm{SU}(2,1)$ stabilizing a fixed isotropic line $\ell$ in $\mathbb{C}^{2,1}$. Let $(M,H,J)$ be a CR $3$-manifold and let $(\mathcal{G}\to M,\omega)$ be the canonical Cartan geometry of type $(G,P)$ corresponding to the CR structure on $M$. If $\mathbb{V}$ is an irreducible representation of $\mathrm{SU}(2,1)$ then the bundle $\mathcal{V}=\mathcal{G}\times_P \mathbb{V}$ is called a \emph{CR tractor bundle}. Note that every irreducible representation $\mathbb{V}$ of $\mathrm{SU}(2,1)$ is contained in some tensor representation constructed from $\mathbb{C}^{2,1}$ and $(\mathbb{C}^{2,1})^*$ as a subspace of tensors satisfying certain symmetries and the trace-free condition. It follows that knowledge of the so called \emph{(CR) standard tractor bundle} $\mathcal{T}=\mathcal{G}\times_P \mathbb{C}^{2,1}$ is sufficient to recover all of the tractor bundles. The standard tractor bundle $\mathcal{T}\to M$ should be thought of as a $P$-vector bundle, which is equivalent to saying that it is canonically equipped with a signature $(2,1)$ Hermitian bundle metric (since $P\subset \mathrm{SU}(2,1)$) and that the fibers of $\mathcal{T}$ are canonically filtered vector spaces
\begin{equation*}
\mathcal{T}_x^1 \subset \mathcal{T}_x^0 \subset  \mathcal{T}_x, \qquad x\in M 
\end{equation*}
where $\mathcal{T}_x^1$ is an isotropic line and $\mathcal{T}_x^0 = (\mathcal{T}_x^1)^{\perp}$ (since $P$ preserves the filtration $\ell \subset \ell^{\perp} \subset \mathbb{C}^{2,1}$). The $P$-principal Cartan bundle $\mathcal{G}\to M$ may readily be recovered from the standard tractor bundle as the bundle of $P$-adapted frames, that is, frames where the first frame vector is chosen from $\mathcal{T}^1$, the second from $\mathcal{T}^0$, and the frame is normalized so that the signature $(2,1)$ bundle metric takes the form
\begin{equation*}
\left(\begin{array}{ccc}
0 & 0 & 1\\
0 & 1 & 0\\
1 & 0 & 0
\end{array}\right).
\end{equation*}
Moreover, the canonical CR Cartan connection $\omega$ on $\mathcal{G}\to M$ may equivalently be viewed as a linear connection $\nabla$ on $\mathcal{T}\to M$, called the \emph{tractor connection}, which preserves the bundle metric on $\mathcal{T}\to M$. The tractor connection on the standard tractor bundle induces a linear connection on each tractor bundle in the obvious way. As in \cite{GoverGraham2005, CE2018-obstruction-flatI} we construct $(\mathcal{T},\nabla)$ without reference to $(\mathcal{G},\omega)$.

Following \cite{GoverGraham2005} we take $\cT$ to be the set of equivalence classes of pairs $(\theta, (\sigma,\mu^1,\rho))$, where $\theta$ is a contact form and $(\sigma,\mu^1,\rho)\in \mathcal{E}(0,1)\oplus\mathcal{E}^{1}(-1,0)\oplus\mathcal{E}(-1,0)$, under the equivalence relation: $(\theta, (\sigma,\mu^1,\rho))\sim (\thetah, (\hat{\sigma},\hat{\mu}^1,\hat{\rho}))$ if $\thetah=e^{\Ups}\theta$ and
\begin{equation}\label{TractorTransformation}
\left(\begin{array}{c}
\hat{\sigma}\\
\hat{\mu}^1\\
\hat{\rho}
\end{array}\right)
=
\left(
\begin{array}{ccc}
1 & 0 & 0\\
\Ups^{1} & 1 & 0 \\
-\frac{1}{2}(\Ups^{1}\Ups_{1}-i\Ups_0) & - \Ups_{1} & 1
\end{array}
\right)\left(\begin{array}{c}
\sigma\\
\mu^{1}\\
\rho
\end{array}\right)
\end{equation}
where $\Ups_1=\nabla_1 \Ups$, $\Ups^1=\bh^{1\oneb}\Ups_{\oneb}$ with $\Ups_{\oneb}=\nabla_{\oneb} \Ups$, and $\Ups_0 = \nabla_0 \Ups$. The canonical filtration of $\mathcal{T}$ is immediately evident, fixing a contact form $\theta$ this is given by
\begin{equation*}
\mathcal{T}^1
=
\left\{\left(\begin{array}{c}
0\\
0\\
*
\end{array}\right)\right\}
\subset
\mathcal{T}^0
=
\left\{\left(\begin{array}{c}
0\\
*\\
*
\end{array}\right)\right\}
\subset
\mathcal{T}.
\end{equation*}
If $(\theta, (\sigma,\mu^1,\rho))\sim (\thetah, (\hat{\sigma},\hat{\mu}^1,\hat{\rho}))$ then one easily checks that
\begin{equation*}
 2\hat{\sigma}\hat{\rho} + \hat{\mu}^1\hat{\mu}_1 = 2\sigma\rho + \mu^1\mu_1,
\end{equation*}
which defines by polarization a signature $(2,1)$ Hermitian bundle metric $h$ on $\cT$, called the \emph{tractor metric}.
We will adopt the abstract index notation $\cE^A$ for $\mathcal{T}$, and $\cE^{\Ab}$ for $\overline{\mathcal{T}}$, using capitalized Latin letters from the start of the alphabet for our abstract indices. The tractor metric $h$ is then written as $h_{A\Bb}$. Decomposing $\cE^A$ with respect to any choice of contact form $\theta$, we have
\begin{equation}\label{eqn:tractor-metric}
h_{A\Bb} = \left(
\begin{array}{ccc}
0 & 0 & 1\\
0 & \bh_{1\oneb} & 0 \\
1 & 0 & 0
\end{array}
\right).
\end{equation}

The line bundle $\cE(-1,0)$ is naturally included in $\cT$ by the map
\begin{equation*}
\rho \mapsto \left(\begin{array}{c}
0\\
0\\
\rho
\end{array}\right)
\end{equation*}
and is this naturally identified with $\cT^1$. The map $\cE(-1,0)\to\cE^A$ corresponds to a canonical section $\boldsymbol{Z}^A$ of $\cE^A\otimes\cE(1,0)$, known as the \emph{canonical tractor}. Bearing in mind that $\boldsymbol{Z}^A$ is a weighted tractor, we may write
\begin{equation*}
\boldsymbol{Z}^A \overset{\theta}{=} \left(\begin{array}{c}
0\\
0\\
1
\end{array}\right)
\end{equation*}
for any contact form $\theta$. The canonical tractor also induces a canonical projection $\cE^A \to \cE(0,1)$ taking $v^A$ to $\sigma = h_{A\bar{B}}v^A\boldsymbol{Z}^{\bar{B}}$. This corresponds to the obvious projection
\begin{equation*}
\left(\begin{array}{c}
\sigma\\
\mu^1\\
\rho
\end{array}\right)
\mapsto \sigma.
\end{equation*}

If $M$ is a strictly pseudoconvex hypersurface in $\mathbb{CP}^2$, then $\cE(-1,0)$ is the restriction of the tautological line bundle $\mathcal{O}(-1)$ to $M$ and $\cT=\cE^A$ can be identified with the restriction to $M$ of the tautological rank $3$ complex vector bundle over $\mathbb{CP}^2$ (coming from the projection $\mathbb{C}^3\setminus\{0\}\to\mathbb{CP}^2$). The canonical tractor can then be identified with the Euler field on $\mathbb{C}^3$, whence the notation $\boldsymbol{Z}$. From this point of view, however, the origins of the tractor metric $h$ and particularly of the tractor connection are more subtle.

\subsubsection*{The tractor connection}
In order to define the tractor connection, we define the higher order pseudohermitian curvatures 
\begin{equation}\label{eqn:T_1}
T_1 = \frac{1}{12}(\nabla_1 R - 4i\nabla^1A_{11}),
\end{equation}
a section of $\cE_1(-1,-1)$, and the real $(-2,-2)$ density
\begin{equation}\label{eqn:S}
S = -(\nabla^1T_1 + \nabla^{\oneb}T_{\oneb} + \frac{1}{16}R^2 - A^{11}A_{11}).
\end{equation}
With respect to a choice of contact form $\theta$ the tractor connection on a section $v^A \overset{\theta}{=} (\sigma,\mu^1,\rho)$ is then given by
\begin{equation}\label{TractorConnection1}
\nabla_1 v^{A}\overset{\theta}{=}
\left(\begin{array}{c}
\nabla_1\sigma\\
\nabla_1\mu^1 + \rho + \frac{1}{4} R \sigma\\
\nabla_1\rho-iA_{11}\mu^1-\sigma T_1
\end{array}\right),
\end{equation}
\begin{equation}\label{TractorConnection1b}
\nabla_{\oneb}v^{A}\overset{\theta}{=}
\left(\begin{array}{c}
\nabla_{\oneb}\sigma-\mu_{\oneb}\\
\nabla_{\oneb}\mu^1-iA_{\oneb}{^1}\sigma\\
\nabla_{\oneb}\rho-\frac{1}{4} R\mu_{\oneb}+\sigma T_{\oneb}
\end{array}\right),
\end{equation}
and
\begin{equation}\label{TractorConnection0}
\nabla_{0} v^{A}\overset{\theta}{=}
\left(\begin{array}{c}
\nabla_{0}\sigma-\frac{i}{12}R\sigma+i\rho\\
\nabla_{0}\mu^1+\frac{i}{6}R\mu^1-2i\sigma T^1\\
\nabla_{0}\rho-\frac{i}{12}R\rho-2iT_1\mu^1-iS\sigma
\end{array}\right).
\end{equation}
One can verify that these formulae give rise to a well-defined CR invariant connection on $\cE^A$ which preserves the tractor metric $h_{A\bar{B}}$ \cite{CapGover2008,GoverGraham2005}.

The tractor curvature $\kappa$ is a $2$-form valued in (trace free skew-Hermitian) endomorphisms of the standard tractor bundle. Given a choice of contact form $\theta$, $\kappa$ may be decomposed into three components $\kappa_{1\oneb}{}_A{}^B$, $\kappa_{10}{}_A{}^B$, and $\kappa_{\oneb 0}{}_A{}^B$, defined by
\begin{align*}
\nabla_1\nabla_{\oneb} v^B -\nabla_{\oneb}\nabla_1 v^B + i \bh_{1\oneb}\nabla_0 v^B & = \kappa_{1\oneb A}{}^B v^A;\\
\nabla_1\nabla_{0} v^B - \nabla_{0}\nabla_1 v^B - A^{\oneb}{_{1}}\nabla_{\oneb}v^B &= \kappa_{10 A}{}^B v^A;\\
\nabla_{\oneb}\nabla_{0} v^B - \nabla_{0}\nabla_{\oneb} v^B - A^{1}{_{\oneb}}\nabla_{1}v^B &= \kappa_{\oneb 0A }{}^B v^A
\end{align*}
for any section $v^A$ of $\cE^A$ (the tractor connection is coupled with the Tanaka-Webster connection of $\theta$ in order to define the iterated covariant derivatives). By definition the component $\kappa_{1\oneb A}{}^B$ of the tractor curvature is a CR invariant, i.e.\ it does not depend on the choice of $\theta$. However, a straightforward calculation shows that $\kappa_{1\oneb}{}_A{}^B=0$. The vanishing of $\kappa_{1\oneb}{}_A{}^B$ implies that $\kappa_{10 A}{}^B$ and $\kappa_{\oneb 0 A}{}^B$ are CR invariant (this phenomenon is special to $3$-dimensional CR structures, cf.\ \cite{ChernMoser1974}). A straightforward calculation using the above formulae for the tractor connection, the formulae \cref{eqn:commuting-11bar-derivatives-on-densities,eqn:commuting-10-derivatives-on-densities} for the curvature of the density line bundles, and the definitions of $T_1$ and $S$, gives (cf.\ \cite{GoverGraham2005})
\begin{equation}\label{eqn:tractor-curvature}
\ka_{10 A}{}^B v^A \overset{\theta}{=} \left(
\begin{array}{ccc}
0 & 0 & 0\\
0 & 0 & 0 \\
Y_1 & iQ_{11} & 0
\end{array}
\right)
\left(\begin{array}{c}
\sigma\\
\mu^1\\
\rho
\end{array}\right)
\end{equation}
where $Q_{11}$ is given by \cref{eqn:Cartan-umbilical-tensor}, and (cf.\ \cite{CE2018-obstruction-flatI})
\begin{equation}\label{eqn:Y1}
Y_1 = -i\nabla^1 Q_{11}.
\end{equation}
The CR invariance of $Q_{11}$ then follows immediately from the CR invariance of $\kappa_{10 A}{}^B$ and the transformation law \cref{TractorTransformation}. On the other hand, $Y_1$ is not CR invariant, rather the transformation law \cref{TractorTransformation} implies that if $\thetah = e^{\Ups}\theta$ then $\hat{Y}_1 = Y_1 - Q_{11}\Upsilon^1$. (The pair $Q_{11}$ and $Y_1$, respectively, are highly analogous to the Weyl curvature and Cotton tensor in $4$-dimensional conformal geometry.) Since the tractor connection preserves the tractor metric we have $\kappa_{\oneb 0 A}{}^B = - h_{A\bar{D}} h^{B\bar{C}} \overline{\kappa_{10 C}{}^D}$, giving
\begin{equation}\label{eqn:tractor-curvature-oneb0}
\ka_{\oneb 0 A}{}^B v^A \overset{\theta}{=} \left(
\begin{array}{ccc}
0 & 0 & 0\\
iQ_{\oneb}{}^1 & 0 & 0 \\
-Y_{\oneb} & 0 & 0
\end{array}
\right)
\left(\begin{array}{c}
\sigma\\
\mu^1\\
\rho
\end{array}\right)
\end{equation}
where $Y_{\oneb} = \overline{Y_1}$.

\subsection{The adjoint tractor bundle and the obstruction as a divergence}
Let $\mathbb{C}^{2,1}$ denote $\mathbb{C}^3$ equipped with the signature $(2,1)$ Hermitian inner product
\begin{equation*}
\langle (z_0,z_1,z_2), (w_0,w_1,w_2)\rangle = z_0\overline{w_2} + z_1\overline{w_1} + z_2\overline{w_0}
\end{equation*}
chosen so that the standard first and last basis vectors are isotropic. Let $G=\mathrm{SU}(2,1)$ be the linear group preserving the inner product, with Lie algebra
\begin{equation*}
\mathfrak{su}(2,1)=\left\{
\left(
\begin{array}{ccc}
a & z & i\phi\\
w & -2i\,\mathrm{Im}\, a  & -\bar{z} \\
i\psi & -\bar{w} & -\bar{a}
\end{array}
\right) : \phi,\psi\in \mathbb{R},\; a,z,w\in \mathbb{C} \right\}.
\end{equation*}

The adjoint tractor bundle is the bundle induced from $\cG$ by the adjoint representation of $G$ on its Lie algebra $\mathbb{V}=\fg$. Since $\fg$ consists of the trace-free skew-Hermitian endomorphisms of $\mathbb{C}^{2,1}$, the adjoint tractor bundle $\cA\to M$ is the subbundle of $\mathrm{End}(\cT)$ consisting of trace-free skew-Hermitian endomorphisms of $\cT$. A section $s\in \Gamma(\cA)$ may be written with respect to a choice of contact form $\theta$ as
\begin{equation*}
 s_A{}^B \overset{\theta}{=}\left(\begin{array}{ccc}
\mu & \upsilon_1 & iu\\
\nu^1 & -2i \mathrm{Im}\mu& -\upsilon^1 \\
i\lambda & -\nu_1 & -\overline{\mu}
\end{array}\right).
\end{equation*}
If $\thetah=e^{\Ups}\theta$ then
\begin{equation*}
 s_A{}^B \overset{\thetah}{=}
\left(\!
\begin{array}{ccc}
1 & 0 & 0\\
\Ups^{1} & 1 & 0 \\
-\frac{1}{2}(\Ups^{1}\Ups_{1}-i\Ups_0) & - \Ups_{1} & 1
\end{array}
\right) \!
 \left(\begin{array}{ccc}
\mu & \upsilon_1 & iu\\
\nu^1 & -2i \mathrm{Im}\mu& -\upsilon^1 \\
i\lambda & -\nu_1 & -\overline{\mu}
\end{array}\right)\!
\left(\!
\begin{array}{ccc}
1 & 0 & 0\\
-\Ups^{1} & 1 & 0 \\
-\frac{1}{2}(\Ups^{1}\Ups_{1}+i\Ups_0) & \Ups_{1} & 1
\end{array}
\right).
\end{equation*}


The tractor curvature $\kappa$ satisfies the Bianchi identity, $\dee^{\nabla}\kappa=0$, which can be written in terms of the components as $\nabla_{1}\kappa_{\oneb 0A}{}^B - \nabla_{\oneb} \kappa_{10A}{}^B = 0$, i.e.\ 
$$
\nabla^1\kappa_{10A}{}^B = \nabla^{\oneb}\kappa_{\oneb 0A}{}^B.
$$ 
\begin{lemma}\cite[Lemma 7.1]{CE2018-obstruction-flatI}\label{lem:obstruction-flat-as-divergence}
Let $(M,H,J)$ be a CR $3$-manifold. Then the CR obstruction density $\mathcal{O}$ vanishes if and only if $\nabla^1\kappa_{10A}{}^B=0$ (equivalently $\nabla^{\oneb}\kappa_{\oneb 0A}{}^B=0$).
\end{lemma}
\begin{proof}
Fix a background contact form $\theta$. A direct calculation using \cref{eqn:Y1} shows that 
\begin{equation*}
\nabla^1\kappa_{10A}{}^B \overset{\theta}{=} 
\left(\begin{array}{ccc}
0 & 0 & 0\\
0 & 0 & 0 \\
-i(\nabla^1\nabla^1Q_{11} - iA^{11}Q_{11}) & 0 & 0
\end{array}\right).
\end{equation*}
The lemma follows immediately by \cref{eqn:obstruction-density}.
\end{proof}

\subsection{Solving a prolonged system of $\partial_b$-equations}

Let $(M,H,J)$ be a CR $3$-manifold and suppose $u$ is a real weight $(1,1)$ density solving the infinitesimal automorphism equation $\nabla_1\nabla_1 u + iA_{11} u = 0$. Then there exists a section $s_A{}^B$ of the adjoint tractor bundle with $s_A{}^B\boldsymbol{Z}^A\boldsymbol{Z}_B = iu$ satisfying $\nabla_1 s_A{}^B = u\kappa_{10A}{}^B$ and $\nabla_{\oneb} s_A{}^B = u\kappa_{\oneb 0A}{}^B$ \cite{Cap2008,CE2018-obstruction-flatI}. The following gives a holomorphic generalization of this result.
\begin{proposition}\label{prop:holomorphic-prolonged-system}
Suppose $u$ is a complex solution of the CR invariant equation $\nabla_1\nabla_1 u + iA_{11} u =0$ on weight $(1,1)$ densities. Then there exist trace free sections $s_A{}^B$ of $\cE_A{}^B$ with $s_A{}^B\boldsymbol{Z}^A\boldsymbol{Z}_B = iu$ satisfying $\nabla_1 s_A{}^B = u\kappa_{10A}{}^B \mod \boldsymbol{Z}_A\boldsymbol{Z}^B$. 
\end{proposition}
\begin{remark}
The sections $s_A{}^B$ given by this proposition depend on one free parameter, namely a section $\nu^1$ of $\cE^1(-1,-1)$. We expect that for a specific choice of $\nu^1$ the section $s_A{}^B$ constructed in the proof indeed satisfies $\nabla_1 s_A{}^B = u\kappa_{10A}{}^B$ without the need to mod out by $\boldsymbol{Z}_A\boldsymbol{Z}^B$. This would require a substantial calculation however, and the precise equality is not needed for our main results.
\end{remark}
\begin{proof}
We shall mimic what happens in the case where $u$ is real. We start by setting by setting $s_A{}^B\boldsymbol{Z}^A\boldsymbol{Z}_B = iu$, so that
\begin{equation*}
 s_A{}^B \overset{\theta}{=}\left(\begin{array}{ccc}
\mu & \upsilon_1 & iu\\
\nu^1 & \chi-\mu & -\xi^1 \\
i\lambda & -\eta_1 & -\chi
\end{array}\right).
\end{equation*}
Note that if $s_A{}^B$ were to be an adjoint tractor, then $u$ and $\lambda$ would be real, and $\xi_{\oneb}$, $\chi$, $\eta^{\oneb}$ would be $\overline{\upsilon_1}$, $\overline{\mu}$, $\overline{\nu^1}$ respectively. If $s_A{}^B$ is given as above then, with respect to $\theta$, $\nabla_1 s_A{}^B$ equals
\begin{equation*}
\footnotesize{
\left(\begin{array}{ccc}
\nabla_1\mu - \frac{1}{4}R\ups_1 + iuT_1 & \nabla_1\ups_1 - A_{11}u & i\nabla_1 u - \ups_1\\
\nabla_1\nu^1 + i\lambda + \frac{1}{4}R(2\mu-\chi)-\xi^1T_1 & * & -\nabla_1\xi^1 - (2\chi -\mu) + \frac{i}{4}Ru\\
i\nabla_1\lambda - iA_{11}\nu^1 +\frac{1}{4}R\eta_1 -(\mu+\chi)T_1 & -\nabla_1\eta_1 - i (2\chi - \mu)A_{11} - \upsilon_1 T_1 & -\nabla_1\chi + \eta_1 + iA_{11}\xi^1 - iuT_1
\end{array}\right)
}
\end{equation*}
where the entry marked $*$ is determined by the trace-freeness. In order for the top right entry $i\nabla_1 u - \ups_1$ to vanish we must choose $\upsilon_1 = i\nabla_1 u$. The fact that $u$ solves the holomorphic infinitesimal automorphism equation $\nabla_1\nabla_1 u + iA_{11} u =0$ then gives us that $\nabla_1\ups_1 - A_{11}u=0$. As in the case where $u$ is real \cite{CE2018-obstruction-flatI} we define $\mu$ by $\mu=\frac{1}{3}(\nabla_0 u - \nabla^1\ups_1 -\frac{i}{4}uR )$. A straightforward calculation, which is formally the same as in the case where $u$ is real, gives that $\nabla_1\mu - \frac{1}{4}R\ups_1 + iuT_1=0$. We fix $\chi$ and $\eta_1$ by setting $-\nabla_1\xi^1 - (2\chi -\mu) + \frac{i}{4}Ru=0$ and $-\nabla_1\chi + \eta_1 + iA_{11}\xi^1 - iuT_1=0$. Using that $2\chi -\mu= -\nabla_1\xi^1+ \frac{i}{4}Ru$ and $\eta_1 = \nabla_1\chi - iA_{11}\xi^1 + iuT_1$, where $\xi^1$ is as yet undetermined, the bottom middle entry becomes
\begin{multline}\label{eq:basic-equation}
\frac{1}{2}\nabla_1^3\xi^1+2i A_{11}\nabla_1\xi^1+i\xi^1\nabla_1A_{11}-\frac{i}{8}(u\nabla_1^2R+
2\nabla_1R\, \nabla_1u+R\, \nabla_1^2u)\\-\frac{1}{6}\left (\nabla_1^2\nabla_0u-i\nabla_1^2\nabla^1\nabla_1 u
-\frac{i}{4}(u\nabla_1^2R+2\nabla_1R\, \nabla_1u+R\, \nabla_1^2u)\right)\\ -\frac{i}{12}u(\nabla_1^2R -4i\nabla_1\nabla^1 A_{11})+\frac{1}{4}A_{11}Ru-\frac{i}{6}\nabla_1u(\nabla_1R-4i\nabla^1A_{11})
\end{multline}
where $\nabla_1^k$ means $\nabla_1$ to the $k$th power. We need to show that $\xi^1$ can be chosen such that \cref{eq:basic-equation} is equal to $iuQ_{11}$. We first simplify the part of this expression which does not involve $\xi^1$. The idea is to commute derivatives in the expression $\nabla_1^2\nabla_0u-i\nabla_1^2\nabla^1\nabla_1 u$ to obtain terms involving $\nabla_1^2 u$, and use that $\nabla_1^2 u = -iA_{11} u$. Using the commutation formulae \cref{eqn:Ricci-identity,eqn:Tanaka-Webster-commuting-10-derivatives,eqn:commuting-11bar-derivatives-on-densities,eqn:commuting-10-derivatives-on-densities} we have
\begin{equation*}
\nabla_1^2\nabla^1\nabla_1u=\nabla_1\nabla^1\nabla_1^2u-i\nabla_0\nabla_1^2u-iA_{11}\nabla^1\nabla_1u +
i\nabla^1A_{11}\nabla_1u -\nabla_1 R\nabla_1u-R\,\nabla_1^2u,
\end{equation*}
and
\begin{equation*}
\nabla_1^2\nabla_0u=\nabla_0\nabla_1^2u+(\nabla_1\nabla^1u+\nabla^1\nabla_1u)A_{11}
+\nabla^1u\nabla_1 A_{11}-\nabla_1u\nabla^1A_{11}.
\end{equation*}
Thus,
\begin{equation*}
\nabla_1^2\nabla_0u-i\nabla_1^2\nabla^1\nabla_1 u=-i\nabla_1\nabla^1\nabla_1^2u+A_{11} \nabla_1\nabla^1u\\+\nabla^1u\nabla_1A_{11}+ iR\, \nabla_1^2u  +i\nabla_1R\,\nabla_1u.
\end{equation*}
Substituting $\nabla_1^2u=-iA_{11}u$, we obtain 
\begin{equation}\label{eq:finalcomm}
\nabla_1^2\nabla_0u-i\nabla_1^2\nabla^1\nabla_1 u=-\nabla_1u \nabla^1 A_{11} -u\nabla_1\nabla^1A_{11}+ R\, A_{11}u  +i\nabla_1R\,\nabla_1u.
\end{equation}
Substituting \cref{eq:finalcomm} and $\nabla_1^2u=-iA_{11}u$, \cref{eq:basic-equation} becomes
\begin{multline}\label{eq:basic-equation-2}
\frac{1}{2}\nabla_1^3\xi^1+2i A_{11}\nabla_1\xi^1+i\xi^1\nabla_1A_{11} - \frac{i}{6} u \nabla_1^2 R -\frac{i}{2}\nabla_1 R\nabla_1 u -\frac{1}{2}\nabla_1u\nabla^1A_{11} -\frac{1}{6}u\nabla_1\nabla^1A_{11}.
\end{multline}
Again using the case where $u$ is real as a guide, we set $\xi^1 = -i\nabla^1 u$. Let $L\xi^1$ denote $\frac{1}{2}\nabla_1^3\xi^1+2i A_{11}\nabla_1\xi^1+i\xi^1\nabla_1A_{11}$. Then
\begin{equation*}
L\xi^1=\frac{1}{2}\nabla_1^3\xi^1+2i A_{11}\nabla_1\xi^1+i\xi^1\nabla_1A_{11}=-\frac{i}{2} \nabla_1^3\nabla^1u+2A_{11}\nabla_1\nabla^1 u +\nabla^1u\nabla_1 A_{11}.
\end{equation*}
We have also have the commutation formula
\begin{multline*}
\nabla_1^3\nabla^1u=\nabla_1\nabla^1\nabla_1^2u-2i\nabla_0\nabla_1^2u -i A_{11} \nabla_1\nabla^1u\\-2i A_{11} \nabla^1\nabla_1u -R\, \nabla_1^2u  -\nabla_1R\,\nabla_1u
+2i \nabla_1u\nabla^1A_{11}-i \nabla^1u\nabla_1A_{11}.
\end{multline*}
which yields
\begin{multline*}
L\xi^1=-\frac{i}{2}\nabla_1\nabla^1\nabla_1^2u-\nabla_0\nabla_1^2u - A_{11} \nabla^1\nabla_1u +\frac{i}{2}R\, \nabla_1^2u \\ +\frac{i}{2}\nabla_1R\,\nabla_1u
+ \nabla_1u\nabla^1A_{11}+\frac{3}{2}A_{11}\nabla_1\nabla^1 u +\frac{1}{2}\nabla^1u\nabla_1 A_{11}.
\end{multline*}
Substituting $\nabla_1^2u=-iA_{11}u$, we obtain
\begin{multline*}
L\xi^1=-\frac{1}{2}\nabla_1\nabla^1(A_{11}u)+i\nabla_0(A_{11}u) - A_{11} \nabla^1\nabla_1u +\frac{1}{2}R\, A_{11}u \\ +\frac{i}{2}\nabla_1R\,\nabla_1u
+ \nabla_1u\nabla^1A_{11}+\frac{3}{2}A_{11}\nabla_1\nabla^1 u +\frac{1}{2}\nabla^1u\nabla_1 A_{11},
\end{multline*}
so
\begin{multline*}
L\xi^1=-\frac{1}{2}u\nabla_1\nabla^1A_{11}+iu\nabla_0A_{11}+iA_{11}\nabla_0u \\+A_{11}\nabla_1\nabla^1 u - A_{11} \nabla^1\nabla_1u +\frac{1}{2}R\, A_{11}u  +\frac{i}{2}\nabla_1R\,\nabla_1u
+ \frac{1}{2}\nabla_1u\nabla^1A_{11}.
\end{multline*}
Since $\nabla_1\nabla^1u-\nabla^1\nabla_1u=-i\nabla_0u$, we obtain
\begin{equation}\label{eq:L0xi1-2}
L\xi^1=-\frac{1}{2}u\nabla_1\nabla^1A_{11}+iu\nabla_0A_{11} +\frac{1}{2}R\, A_{11}u  +\frac{i}{2}\nabla_1R\,\nabla_1u
+ \frac{1}{2}\nabla_1u\nabla^1A_{11}.
\end{equation}
Substituting \eqref{eq:L0xi1-2} into \eqref{eq:basic-equation-2} we obtain
\begin{equation}
u\left(-\frac{i}{6}\nabla_1^2R +\frac{1}{2}A_{11}R +i\nabla_0A_{11}-\frac{2}{3}\nabla_1\nabla^1 A_{11}\right),
\end{equation}
which equals $iuQ_{11}$. We now choose $\lambda$ so that $\nabla_1\nu^1 + i\lambda + \frac{1}{4}R(2\mu-\chi)-\xi^1T_1=0$, where $\nu_1$ is arbitrary. Since (noting the form of the tractor metric \cref{eqn:tractor-metric})
$$
\boldsymbol{Z}_A\boldsymbol{Z}^B \overset{\theta}{=} 
\left(\begin{array}{ccc}
0 & 0 & 0\\
0 & 0 & 0 \\
1 & 0 & 0
\end{array}\right)
$$
the result follows.

\end{proof}

\begin{remark}\label{rem:holomorphic-sections}
As a side remark, we note that in the embedded case one can find actual CR holomorphic sections of the complexified adjoint tractor bundle. Recall that a contact form $\theta$ on a CR $3$-manifold $(M,H,J)$ is pseudo-Einstein \cite{FeffermanHirachi2003, CaseYang2013} if it is locally volume normalized by a closed $(2,0)$-form. If $\theta$ is volume normalized with respect to a global closed $(2,0)$-form which admits a global cube root $\sigma$, then by \cite[Proposition 4.6]{CaseGover-arxiv2017} there exists a corresponding global CR holomorphic section of the cotractor bundle $I_A$ with $I_A\boldsymbol{Z}^A=\sigma$. If $(M,H,J)$ is compact and embedded in $\mathbb{C}^2$, then any pseudo-Einstein structure gives rise to such a section $I_A$. In this case, one can easily see that there exist $3$ pointwise linearly independent CR holomorphic cotractors $I_A^1$, $I_A^2$, $I_A^3$ (indeed, it suffices to take the cotractors corresponding to $\sigma$, $e^{f_1}\sigma$, $e^{f_2}\sigma$ where $f_1$ and $f_2$ are suitable linear functions on $\mathbb{C}^2$). One therefore obtains a dual basis of CR holomorphic tractors $J^A_1$, $J^A_2$, $J^A_3$. Combining these one obtains global CR holomorphic sections $I_A^jJ^B_k$ of $\cE_A{}^B=\mathrm{End}(\cT)$ and corresponding CR anti-holomorphic sections ${\hat{s}^{(j,k)}}{}_A{}^B=J_A^jI^B_k$, where $j,k=1,2,3$.
\end{remark}

\section{Proofs of main results} \label{sec:main-technical-theorem}

\begin{proof}[Proof of \cref{thm:main-technical-theorem}]
Fix a contact form $\theta$ and let $\nabla$ denote the Tanaka-Webster connection of $\theta$ coupled with the CR tractor connection. By \cref{lem:obstruction-flat-as-divergence}, since $\mathcal{O}=0$ we have $\nabla^{\oneb}\kappa_{\oneb0A}{}^B=0$. 
If $u$ is a weight $(1,1)$ complex solution of $\nabla_1\nabla_1 u + iA_{11} u = 0$, then by \cref{prop:holomorphic-prolonged-system} there exists a trace free section $s_A{}^B$ of $\cE_A{}^B$ with $s_A{}^B\boldsymbol{Z}^A\boldsymbol{Z}_B=iu$ satisfying $\nabla_1 s_A{}^B = u\kappa_{10A}{}^B + r_1\boldsymbol{Z}_A\boldsymbol{Z}^B$ where $r_1\in \Gamma(\cE_1(-1,-1))$. Note that $s_A{}^{C}s_B{}^A\,\nabla^{\oneb}\kappa_{\oneb 0C}{}^B = s_A{}^C \,(\nabla^{\oneb}\kappa_{\oneb 0B}{}^A) \,s_C{}^B$ is a density of weight $(-2,-2)$ so can be invariantly integrated. Integrating by parts we obtain
\begin{align}\label{eqn:int-by-parts}
0=\int s_A{}^C \,(\nabla^{\oneb}\kappa_{\oneb 0B}{}^A) \,s_C{}^B & = - \int \left( (\nabla^{\oneb}s_A{}^C)\,\kappa_{\oneb 0B}{}^A \,s_C{}^B + s_A{}^C \,\kappa_{\oneb 0B}{}^A \,\nabla^{\oneb} s_C{}^B \right)\\
\nonumber &= - \int u h^{1\oneb} \left( \kappa_{10A}{}^C \,\kappa_{\oneb 0B}{}^A\, \,s_C{}^B + s_A{}^C \,\kappa_{\oneb 0B}{}^A \,\kappa_{10C}{}^B \right)
\end{align}
where we used that $\boldsymbol{Z}_A \kappa_{\oneb 0B}{}^A =0$ and $\kappa_{\oneb 0B}{}^A \boldsymbol{Z}^B=0$. Now by \cref{eqn:tractor-curvature} and \cref{eqn:tractor-curvature-oneb0} we have
\begin{equation*}
h^{1\oneb}\kappa_{\oneb 0B}{}^A \,\kappa_{10C}{}^B  \overset{\theta}{=}
\left(
\begin{array}{ccc}
0 & 0 & 0\\
iQ_{\oneb}{}^1 & 0 & 0 \\
-Y_{\oneb} & 0 & 0
\end{array}
\right)
\left(
\begin{array}{ccc}
0 & 0 & 0\\
0 & 0 & 0 \\
Y^{\oneb} & iQ_{1}{}^{\oneb} & 0
\end{array}
\right)
=\left(
\begin{array}{ccc}
0 & 0 & 0\\
0 & 0 & 0 \\
0 & 0 & 0
\end{array}
\right)
\end{equation*}
and 
\begin{equation*}
h^{1\oneb}\kappa_{10A}{}^C \,\kappa_{\oneb 0B}{}^A  \overset{\theta}{=}
\left(
\begin{array}{ccc}
0 & 0 & 0\\
0 & 0 & 0 \\
Y^{\oneb} & iQ_{1}{}^{\oneb} & 0
\end{array}
\right)
\left(
\begin{array}{ccc}
0 & 0 & 0\\
iQ_{\oneb}{}^1 & 0 & 0 \\
-Y_{\oneb} & 0 & 0
\end{array}
\right)
=\left(
\begin{array}{ccc}
0 & 0 & 0\\
0 & 0 & 0 \\
- |Q|^2 & 0 & 0
\end{array}
\right)
\end{equation*}
where $|Q|^2= Q_{11}Q^{\oneb\oneb}$. Hence \cref{eqn:int-by-parts} simplifies to
\begin{equation*}
\int u h^{1\oneb}\kappa_{10A}{}^C \,\kappa_{\oneb 0B}{}^A\, \,s_C{}^B = 0
\end{equation*}
and we have
\begin{align*}
h^{1\oneb}\kappa_{10A}{}^C \,\kappa_{\oneb 0B}{}^A\, \,s_C{}^B  \,&\overset{\theta}{=}\,
\mathrm{tr} \, \left(
\begin{array}{ccc}
0 & 0 & 0\\
0 & 0 & 0 \\
- |Q|^2 & 0 & 0
\end{array}
\right) \left(\begin{array}{ccc}
\mu & \upsilon_1 & iu\\
\nu^1 & -2i \mathrm{Im}\mu& -\upsilon^1 \\
i\lambda & -\nu_1 & -\overline{\mu}
\end{array}\right) \\
&=\, \mathrm{tr} \left(
\begin{array}{ccc}
0 & 0 & 0\\
0 & 0 & 0 \\
-\mu |Q|^2  & -\ups_1|Q|^2 & -iu|Q|^2
\end{array}
\right) 
= -iu|Q|^2.
\end{align*}
We conclude that $\int u^2|Q|^2 =0$.
\end{proof}

\begin{corollary}\label{cor:main-technical-corollary}
Let $(M,H,J)$ be a compact strictly pseudoconvex CR $3$-manifold which is obstruction flat. For any two weight $(1,1)$ complex solutions $u_1$, $u_2$ of the infinitesimal automorphism equation and any anti-CR function $f$, we have
$$
\int_M f u_1 u_2 |Q|^2 = 0.
$$
\end{corollary}
\begin{proof}
Note that $fu_1$ is also a complex solution of the infinitesimal automorphism equation, so without loss of generality we may assume $f=1$. Since $u_1+u_2$ is also a complex solution, the result follows from \cref{thm:main-technical-theorem}.
\end{proof}

\begin{remark}
Note that in the proof of \cref{thm:main-technical-theorem} it is possible to replace the integrand $s_A{}^C \,(\nabla^{\oneb}\kappa_{\oneb 0B}{}^A) \,s_C{}^B$ by $s_A{}^C \,(\nabla^{\oneb}\kappa_{\oneb 0B}{}^A) \,\hat{s}^{(j,k)}{}_C{}^B$ where $\hat{s}^{(j,k)}{}_C{}^B$ is as in \cref{rem:holomorphic-sections}. This leads directly to a result which turns out to be a special case of \cref{cor:main-technical-corollary}. 
\end{remark}

\begin{proof}[Proof of \cref{thm:main-theorem2}]
Note that the condition $(\mathrm{Im}\,u)^2<(\mathrm{Re}\,u)^2$ almost everywhere is equivalent to $\mathrm{Re}\,u^2 > 0$ almost everywhere. Hence, \cref{thm:main-technical-theorem} implies $|Q|^2=0$ and hence $(M,H,J)$ is locally spherical.
\end{proof}

\bibliographystyle{abbrv}

\newcommand{\noopsort}[1]{}

\end{document}